\documentclass{amsart}

\pdfoutput=1

\usepackage[utf8]{inputenc}

\usepackage{amssymb}
\usepackage[pdftex]{graphicx}

\usepackage{enumerate}

\theoremstyle{theorem}
\newtheorem{theorem}{Theorem}[section]
\newtheorem{proposition}[theorem]{Proposition}
\newtheorem{lemma}[theorem]{Lemma}

\theoremstyle{definition}
\newtheorem{definition}[theorem]{Definition}
\newtheorem{example}[theorem]{Example}
\newtheorem{remark}[theorem]{Remark}

\def\End{\mathop{\rm End}\nolimits}

\def\ZZ{\mathbb{Z}}
\def\QQ{\mathbb{Q}}
\def\II{\mathbb{I}}
\newcommand{\plocal}{\mathbb{Z}_{(p)}}

\newcommand{\id}{\mathbf{1}}
\newcommand{\0}{\mathbf{0}}

\newcommand{\ihom}{\sharp}
\newcommand{\wprod}{\bigoplus}
\newcommand{\comp}{\,}

\def\im{\mathop{{\mathrm{im}}}}

\def\Tel{\mathop{{\mathrm{Tel}}}}
\def\Hom{{{\mathrm{Hom}}}}
\def\Ann{{{\mathrm{Ann}}}}

\def\ldiv{\mathbin{{\!\smallsetminus\!}}}
\def\rdiv{\mathbin{{\!\reflectbox{$\smallsetminus$}\!}}}

\begin{document}


  \author{Damir Franeti\v{c}}
  \address{Univerza v Ljubljani \\ Fakulteta za ra\v cunalni\v stvo in informatiko \\ Ve\v{c}na pot 113 \\ 1000 Ljubljana \\ Slovenia}
	\email{damir.franetic@fri.uni-lj.si}
  \author{Petar Pave\v{s}i\'{c}}
  \address{Univerza v Ljubljani \\ Fakulteta za matematiko in fiziko \\ Jadranska 19 \\ 1111 Ljubljana \\ Slovenia}
  \email{petar.pavesic@fmf.uni-lj.si}
  \thanks{The second author was partly supported by the Slovenian Research Agency grant P1-0292-0101, project No. J1-6721-0101.}
  \title{Loop near-rings and unique decompositions of H-spaces}
  \begin{abstract}
  	For every H-space $X$ the set of homotopy classes $[X,X]$ possesses a natural algebraic structure of a loop near-ring.
	Albeit one cannot say much about general loop near-rings, it turns out that those that arise from H-spaces are
	sufficiently close to rings to have a viable Krull--Schmidt type decomposition theory, which is then reflected
	into decomposition results of H-spaces. In the paper we develop the algebraic theory of local loop near-rings
	and derive an algebraic characterization of indecomposable and strongly indecomposable H-spaces. As a consequence,
	we obtain unique decomposition theorems for products of H-spaces. In particular, we are able to treat certain
	infinite products of H-spaces, thanks to a recent breakthrough in the Krull--Schmidt theory for infinite products.
	Finally, we show that indecomposable finite $p$-local H-spaces are automatically strongly indecomposable, which
	leads to an easy alternative proof of classical unique decomposition theorems of Wilkerson and Gray.
  \end{abstract}
  \keywords{H-space, near-ring, algebraic loop, idempotent, strongly indecomposable space, Krull--Schmidt--Remak--Azumaya theorem}
  \subjclass[2010]{55P45, 16Y30}

\maketitle

\section*{Introduction}
In this paper we discuss relations between unique decomposition theorems in algebra and homotopy theory. 
Unique decomposition theorems usually state that sum or product decompositions (depending on the category),
whose factors are strongly indecomposable, are essentially unique. The standard algebraic example is the 
Krull--Schmidt--Remak--Azumaya theorem. In its modern form the theorem states that any decomposition of an $R$-module
into a direct sum of indecomposable modules is unique, provided that the endomorphism rings of the summands are local rings
(see \cite[theorem 2.12]{Facchini}). Modules with local endomorphism rings are said to be {\em strongly indecomposable} 
and they play a pivotal role in the study of cancellation and unique decomposition of modules. 
For example, every indecomposable module of finite length is strongly indecomposable which implies the classical Krull--Schmidt 
theorem (see \cite[lemma 2.21 and corollary 2.23]{Facchini}).

Similar results on unique decompositions have been obtained by P. Freyd \cite{Freyd} and H. Margolis \cite{Margolis} in stable homotopy category,  and 
by C. Wilkerson \cite{Wilkerson} and B. Gray \cite{Gray} in unstable homotopy category. However, even when their arguments closely parallel standard algebraic 
approach, the above authors choose to rely on specific properties of topological spaces, and avoid reference to purely algebraic results. 
In \cite{Pav} the second author considered factorizations in stable homotopy category from the algebraic viewpoint.
He first pointed out that the endomorphism rings of stable $p$-complete spectra are finite $\widehat{\mathbb{Z}}_p$-algebras, and those are
known to be semiperfect (see~\cite[example 23.3]{Lam}). The unique decomposition for finite $p$-complete spectra then follows immediately because 
the Krull--Schmidt--Remak--Azumaya theorem holds for modules whose endomorphism ring is semiperfect. 

The $p$-local case is more difficult, but Pave\v{s}i\'{c} was able to show (see \cite[theorem 2.1]{Pav}) that the endomorphism rings of finite $p$-local 
spectra are also semiperfect rings, which implies that finite $p$-local spectra decompose uniquely. 
The efficiency of the algebraic approach motivated our attempt to derive unique decomposition theorems  in unstable homotopy category. 
The leading idea is that the set $[X,X]$ of homotopy classes of self-maps of $X$ should play a role in the decomposition theory of H-spaces 
that is analogous to the role of endomorphism rings in the decomposition of modules. However, the situation is more complicated 
because of the fact that for a general H-space $X$ the set $[X,X]$ is not a ring but possesses only the much weaker structure of a loop near-ring.
Thus we were forced to develop first a notion of localness for loop near-rings, and then to characterize H-spaces that are
strongly indecomposable and appear as prime factors in unique decompositions.  
One of the important advantages of our approach is that there are
stronger versions of the Krull--Schmidt--Remak--Azumaya theorem that can be used to derive new decomposition theorems. In particular a recently proven result about 
unique decompositions of infinite products of modules led to new unique decomposition theorems for infinite products of H-spaces, cf. theorems \ref{thm:KSAprod} 
and \ref{thm:KSAwprod} below. 

The paper is organized as follows. In \S\ref{sect:LNRs} we study the set of homotopy classes $\End(X):=[X,X]$ for a connected H-space $X$ and show that it has 
the algebraic structure of a loop near-ring. Since this structure is not well-known, we then recall some basic facts about loop near-rings, 
generalize the concept of localness to loop near-rings and prove the most relevant results. More algebraic details are developed in a forthcoming paper. 
In \S\ref{sect:unique} we define strongly indecomposable H-spaces and show that a decomposition of an H-space as a product of strongly indecomposable 
factors is essentially unique. Finally, in \S\ref{sect:p-local} we prove that for finite, $p$-local H-spaces indecomposable implies strongly indecomposable, 
which in turn yields a unique decomposition theorem for $p$-local H-spaces. 

Our approach can be almost directly dualized to simply-connected coH-spaces and connective CW-spectra. See remarks~\ref{rem:dual1} and~\ref{rem:dual2}.
All spaces under consideration are assumed to be pointed and to  have the homotopy type of a connected CW-complex. Maps and homotopies are 
base-point preserving, but we omit the base points from the notation and do not distinguish between a map and its homotopy class.  

\section{Loops and loop near-rings}
\label{sect:LNRs}
If $X$ is an H-space then the set $[X,X]$ of homotopy classes of self-maps admits two natural 
binary operations. The first - \emph{multiplication} - is induced by the composition $f g$ of 
maps $f,g \colon X \to X$: it is associative with the identity map $\id_X \colon X\to X$ acting as the neutral element, 
so the resulting algebraic structure $([X,X],\cdot)$ is a monoid. 
The second operation - \emph{addition} - is induced by the H-structure: it is in general neither commutative nor associative, 
and the constant map $\0_X \colon X\to X$ represents the neutral element.
If the H-space $X$ is connected, then $([X,X],+)$ is a so called (algebraic) loop (see \cite[theorem 1.3.1]{Zabrodsky}). 
Moreover, addition and composition on $[X,X]$ are related by right 
distributivity, i.e. $(f + g) h = f h + g h$ holds for all $f,g,h \colon X \to X$. The resulting algebraic structure 
$\End(X):=([X,X],+,\cdot)$ is called a 
\emph{(right) loop near-ring}, a structure that was first introduced by Ramakotaiah~\cite{Ramakotaiah}. We are not aware 
of any papers on loop near-rings that arise in topology. However, if $X$ is an H-group, then 
$\End(X)$ is a near-ring, and this stronger structure  has already been studied by Curjel~\cite{Curjel}, and more recently by 
Baues~\cite{Baues} and others. 

\subsection{Basic properties}
	The definition of a loop near-ring is similar to that of a ring but it lacks some important ingredients: addition is not 
	required to be commutative nor associative, and only one
	of the distributivity laws is present. The resulting structure is often very different from a ring but nevertheless, 
	a surprising number of concepts and facts from ring theory 
	can be suitably extended to this more general setting. We recall the definitions and state relevant results. 
	
	\begin{definition}
		\emph{An (Algebraic) loop} consists of a set $G$ equipped with a binary operation $+$ satisfying the following properties: 
			\begin{itemize}
				\item for every $a,b\in G$ the equations $a+x=b$ and $y+a=b$ have unique solutions $x,y \in G$; 
				\item there exists a two-sided zero, i.e. an element $0 \in G$ such that $0+a=a+0=a$ for all $a \in G$.
			\end{itemize}
	\end{definition}
	A loop is essentially a `non-associative group'. Existence of unique solutions to 
	equations implies that left and right cancellation laws hold in a loop. We can define the operations of {\em left} and 
	{\em right difference} $\ldiv$ and $\rdiv$ 
	where $x = a \ldiv b$ is the unique solution of the equation $a+x=b$, and $y= b \rdiv a$ is the unique solution of  the equation $y+a=b$. 
	
	A {\em loop homomorphism} is a function $\phi \colon G \to H$ between loops $G$ and $H$ such that $\phi(a+b) = \phi(a)+\phi(b)$ 
	for all $a, b \in G$. 
	Since $\phi(0) = \phi(0)+\phi(0)$, the cancellation in $H$ gives $\phi(0)=0$. Similarly we get $\phi(a \ldiv b) = \phi(a) \ldiv \phi(b)$ 
	and $\phi(a \rdiv b) = \phi(a) \rdiv \phi(b)$. 

	As in the theory of groups we can define two kinds of subobjects,  subloops and normal subloops. A subset of a loop $G$ is a 
	{\em subloop} of $G$ if it is closed with respect to the addition and both difference
	operations. A direct definition of a normal subloop is more complicated, as we must take into account the non-associativity of 
	the addition: a subloop $K \le G$ is a {\em normal subloop} if for all $a,b\in G$ we have
	$$a+K = K+a,\ \ (a+b)+K = a+(b+K)\ \  \textrm{and}\ \  (K+a)+b = K+(a+b).$$
	We often use a slicker characterization: a subset of $G$ is a subloop if it is the image of some loop homomorphism; it is 
	a normal subloop if it is a kernel of some loop homomorphism.
	See~\cite[chapter IV]{Bruck} for a detailed treatment of these concepts.
	
	\begin{definition}
		A (right) {\em loop near-ring} $(N,+,\cdot)$ consists of a set $N$ with two binary operations $+$ and $\cdot$ such that:
		\begin{itemize}
			\item $(N,+)$ is a loop, 
			\item $(N, \cdot)$ is a monoid,
			\item multiplication $\cdot$ is  right distributive over addition $+$ and $n 0=0$ holds for every $n\in N$.  
		\end{itemize}
 		If $(N, +)$ is a group, $(N,+, \cdot)$ is a {\em near-ring}.
	\end{definition}
	We have slightly departed from the definition of a loop near-ring in \cite{Ramakotaiah} by requiring that there exists a neutral element for the multiplication, and that $N 0=0$. 
	This modification is motivated by the fact that $\End(X)$ is always unital and the constant map $0$ satisfies the property $0 n=n 0=0$. Note that $0n=0$ follows 
	from the right-distributivity and cancellation, while the symmetric relation $n 0=0$ in \cite{Ramakotaiah} characterizes the so-called zero-symmetric loop near-rings. 
	Let us also remark that if $X$ is a simply-connected coH-space then $\End(X)$ turns out to be a left loop near-ring.		
		
	A generic example of a right near-ring is the near-ring $M(G)$ of {\em all} functions $f \colon G \to G$ from a group $G$ to itself.
	Moreover, if $G$ is only a loop then $M(G)$ is a loop near-ring~\cite[example 1.2]{Ramakotaiah}. The following topological examples are more relevant to our discussion.

	\begin{example}
		\label{ex:nonc}
		To present an example of a near-ring whose additive group is not commutative we first need the following general observation. Given an H-space $X$ with the multiplication map $\mu$, 
		and an arbitrary space $Z$, the sum of maps $f,g\colon Z\to X$ is given by the composition $f+g:=\mu\comp (f\times g)\comp \Delta$ as in the diagram
		$$f+g: Z\xrightarrow{\Delta} Z\times Z 	\xrightarrow{f\times g} X\times X\xrightarrow{\mu} X$$
		This operation is commutative for all spaces $Z$ if and only if $p_1 + p_2 = p_2 + p_1$ holds for the two projections $p_1, p_2 \colon X\times X \to X$ 
		in $[X \times X, X]$. Indeed, one can directly check that $f+g = (p_1+p_2)\comp (f,g)$, and $g+f = (p_2+p_1)\comp (f,g)$, so if $p_1+p_2=p_2+p_1$ then 
		$f+g=g+f$ for every $Z$ and every $f,g\colon Z\to X$. 
		
		A well-known example of an H-structure that is not homotopy commutative is given by the quaternion multiplication on the $3$-sphere $S^3$~\cite{James1}. 
		By the above remark it follows that $[S^3 \times S^3, S^3]$ is a non-abelian group, hence $\End(S^3 \times S^3)$ is a (right) near-ring but not a ring. 
	\end{example}	
	\begin{example}
		Similarly as in the previous example one can show that, given an H-space $X$, the addition on $[Z,X]$ is associative for all spaces $Z$ if and only if 
		the relation $p_1 + (p_2+p_3) = (p_1+p_2)+p_3$ holds for the three projections in $[X \times X \times X, X]$. The octonion multiplication on the sphere $S^7$ is
		a familiar example of an H-structure that is not homotopy associative~\cite{James2}, so the addition in 
		$[S^7\times S^7\times S^7,S^7]$ is not associative. We conclude that
		$\End(S^7 \times S^7 \times S^7)$ is not a near-ring but only a (right) loop near-ring.
	\end{example}	
	\begin{example}
		Our final example is a left loop near-ring induced by a coH-space structure. 
		Every element $\gamma \colon S^6 \to S^3$ of order $3$ in the group $\pi_6(S^3) \cong \ZZ/12$ is a coH-map, therefore its mapping cone 
		$C := S^3 \cup_\gamma e^7$ is a coH-space. Ganea~\cite[proposition 4.1]{Ganea} has proved that $C$ does not admit any associative 
		coH-structures, so in particular the addition induced by the coH-structure in $[C,C\vee C\vee C]$ is not associative. It follows that 
		$\End(C \vee C \vee C)$ is a (left) loop near-ring but not a near-ring.
	\end{example}
	
	\subsection{Local loop near-rings}
	The crucial ingredient in the proof of the Krull--Schmidt--Remak--Azumaya theorem is the assumption that there is a factorization of the given
	module as a direct sum of factors whose endomorphism rings are local. In order to extend this approach to factorizations of H-spaces we 
	need a suitable definition of local loop near-rings. 
	Local near-rings were introduced by Maxson in~\cite{Maxson}. We use the characterization~\cite[theorem 2.8]{Maxson} to extend this 
	concept to loop near-rings.
	A subloop $I \le N$ is said to be an \emph{$N$-subloop} if $N I \subseteq I$. The role of $N$-subloops in the theory of loop near-rings is 
	analogous to that of ideals in rings.
	
	\begin{definition}
		A loop near-ring $N$ is {\em local} if it has a unique maximal $N$-subloop $J \lneq N$.
	\end{definition}
	
	Let $U(N)$ denote the \emph{group of units} of the loop near-ring $N$, that is to say, the group of invertible elements of the monoid 
	$(N,\cdot)$. A function $\phi \colon N \to N'$ is a {\em homomorphism} of loop near-rings if $\phi(1) = 1$,  
	$\phi(m+n) = \phi(m) + \phi(n)$, and $\phi(mn) = \phi(m) \phi(n)$ hold for all $m,n \in N$. Clearly $\phi(0)=0$, and, if $u\in U(N)$, then 
	$\phi(u)\in U(N')$. A homomorphism is said to be \emph{unit-reflecting} if the converse holds, i.e. if $\phi(n)\in U(N')$ implies 
	$n\in U(N)$. 	
	One of the most remarkable properties of loop near-rings that arise in homotopy theory 	is that they come equipped with a unit-reflecting  
	homomorphism into a ring (namely, with the representation into endomorphism of homotopy or homology groups,
	that is unit-reflecting as a consequence of the Whitehead theorem). It is important to observe that the image of such 
	a homomorphism is always 
	a subring of the codomain. The main properties of local loop-near rings are collected in the following theorem. 
		
	\begin{theorem}	\label{thm local}	\label{thm:local_lnr} \
		\begin{enumerate}[(i)]
			\item In a local loop-near ring $N$ the only idempotents are $0$ and $1$.
			\item 
				A loop near-ring $N$ is local if and only if $N \setminus U(N)$ is an $N$-subloop in $N$. 
				Moreover, in this case $N \setminus U(N)$ is the unique maximal $N$-subloop.
			\item Let $\phi\colon N \to R$ be a non-trivial and unit-reflecting homomorphism  from a loop-near ring $N$ to a ring $R$. If $N$ 
				is local then $\im \phi$ is a local subring of $R$. Conversely, if $R$ is local, then $N$ is a local loop near-ring.
		\end{enumerate}
	\end{theorem}
	\begin{proof}
		(i) Let $e = e^2 \in N$ be an idempotent and write an element $n \in N$ as $n = y + ne$. Multiplying this equation by $e$ from the right we get 
		$ne = (y+ne)e = ye + ne$, hence $ye = 0$. Denote by $\Ann(e)$ the {\em annihilator} of $e$, i.e. the subset of all $y \in N$ such that $ye = 0$. 
		We have just seen that $N = \Ann(e)+Ne$. 
		Both subsets, $\Ann(e)$ and $Ne$, are $N$-subloops in $N$ (this is immediate for $Ne$, for $\Ann(e)$ use the fact that $N$ is zero-symmetric). 
		Similarly as for unital rings, Zorn lemma implies that every proper $N$-subloop in $N$ is contained in a maximal 
		$N$-subloop, see~\cite[lemma 2.7]{Maxson}. Clearly, $\Ann(e)$ and $Ne$ cannot both be contained in the unique maximal $N$-subloop 
		$J \lneq N$. Therefore, either $\Ann(e) = N$ or $Ne = N$, which means that either $e = 0$ or $e = 1$.
		
		(ii) Let $N$ be local and let $J \lneq N$ be the unique maximal $N$-subloop. We claim that every $u \in N \setminus J$ has a left inverse. 
		In fact, if $Nu \neq N$, then the $N$-subloop $Nu$ is contained in $J$, hence $u \in J$. Therefore, for $u \in N \setminus J$ we have $Nu = N$, 
		in particular $ku = 1$ for some $k \in N$. Observe that $k \in N \setminus J$ as well. In fact, we have the following chain of 
		implications
		\begin{align*}
			& (1 \rdiv uk)u = u \rdiv uku = u \rdiv u = 0 \\
			\Rightarrow \quad & 1 \rdiv uk \ \textrm{ is not left invertible} \\
			\Rightarrow \quad & uk \in N \setminus J \\
			\Rightarrow \quad & k \in N \setminus J \textrm{. }
		\end{align*}
		We conclude $N \setminus J \subseteq U(N)$. The reverse inclusion 
		$U(N) \subseteq N \setminus J$ is obvious, hence $J = N \setminus U(N)$, which is an $N$-subloop. 
		
		For the reverse implication assume that $N \setminus U(N)$ is an $N$-subloop. 
		Since every proper $N$-subloop $I \lneq N$ is contained in $N \setminus U(N)$, $N \setminus U(N)$ is clearly the unique maximal $N$-subloop.
		
		(iii) Call a subset $K \subseteq N$ an {\em ideal} if $K$ is the kernel of some loop near-ring homomorphism $\psi \colon N \to N'$. Every ideal 
		$K$ is also an $N$-subloop. If $N$ is local with unique maximal $N$-subloop $J$, then $K \subseteq J$ and the quotient $N/K \cong \im \psi$ 
		has $J/K$ as the unique maximal $(N/K)$-subloop. So, in particular, $\im \phi$ is a local ring. 
		
		For the reverse implication, since $\phi$ is unit-reflecting, we have $\phi^{-1}(R \setminus U(R)) = N \setminus U(N)$. 
		As $R$ is a local ring $R \setminus U(R)$ is a left ideal of $R$ by~\cite[theorem 19.1]{Lam}, therefore its preimage $N \setminus U(N)$ is an 
		$N$-subloop of $N$, so by (ii) $N$ is local.
	\end{proof}
\section{Uniqueness of decompositions of H-spaces}
\label{sect:unique}
	\label{sec:uniqueness}
	The classical Krull--Schmidt--Remak--Azumaya theorem says that a factorization of a module as a direct sum of strongly indecomposable 
	modules is essentially unique. 
	In this section we use the theory of loop near-rings to prove an analogous result for product decompositions of H-spaces. 
	
	Given a space $X$ every self map $f\colon X\to X$ induces endomorphisms $\pi_k(f)\in\End(\pi_k(X))$ of the homotopy groups of $X$ 
	that can be combined to obtain the following
	function
	\[
		\beta_X\colon \End(X)\to  \prod_{k=1}^\infty \End(\pi_k(X)), \quad 
			f \mapsto f_\ihom =(\pi_1(f), \pi_2(f), \pi_3(f), \ldots) \textrm{. }
	\]
	A loop near-ring homomorphism $\phi \colon N \to M$ is {\em idempotent-lifting} if, for every idempotent of the form $\phi(n) \in M$ 
	there is an idempotent $e \in N$ such that $\phi(e) = \phi(n)$.
	\begin{proposition} \label{prop beta}
		If $X$ is an H-space then $\beta_X$ is a unit-reflecting and idempotent-lifting homomorphism from a loop near-ring to a ring.
	\end{proposition}
	\begin{proof}
		We already know that $\End(X)$ is a loop near-ring. All homotopy groups of an H-space are abelian 
		so $\End(\pi_k(X))$ are rings, hence the codomain of $\beta_X$ is a direct product of rings. Moreover, $\beta_X$ is a 
		homomorphism of loop near-rings because $(f+g)_\ihom = f_\ihom + g_\ihom$ holds for every H-space $X$, 
		while $(f g)_\ihom = f_\ihom g_\ihom$ by functoriality.   
		To see that $\beta_X$ is unit-reflecting let  $f\colon X\to X$ be such that the induced homomorphism $\beta_X(f)$ is 
		an isomorphism. Then, by the Whitehead theorem, $f$ is a homotopy equivalence, i.e. a unit element of $\End(X)$.
		Finally, if $\beta_X(f)$ is an idempotent, then by~\cite[proposition 3.2]{FraPav} there is a decomposition of $X$ 
		into a product of telescopes 
		$\Tel(f) \times \Tel(f \ldiv \id_X)$. The first factor in this decomposition determines an idempotent $e \colon X \to \Tel(f) \to X$ 
		in $\End(X)$ such that 
		$\beta_X(e)=\beta_X(f)$, proving that $\beta_X$ is idempotent-lifting. 
	\end{proof}
	\begin{remark}\label{rem:dual1}
		All results of this section are easily dualized to simply-connected coH-spaces $X$. As in~\cite{FraPav} one replaces $\pi_*(X)$ 
		with singular homology groups $H_*(X)$ and the homomorphism $\beta_X$ with the homomorphism
		\[
			\alpha_X\colon \End(X)\to  \prod_{k=1}^\infty \End(H_k(X)), \quad 
				f \mapsto f_* =(H_1(f), H_2(f), H_3(f), \ldots) \textrm{. }
		\]
		Product and weak product decompositions of H-spaces are replaced by wedge decompositions of coH-spaces, hence, theorems~\ref{thm:KSAprod} 
		and~\ref{thm:KSAwprod} below are replaced by one dual theorem. Moreover, if one replaces the coH-space $X$ by a connective CW-spectrum $X$, 
		the dualized argument remains the same. Observe that even though $\End(X)$ is a genuine ring in case of CW-spectra, its image under 
		$\alpha_X$ can be easier to understand. 
	\end{remark}
	
	Every decomposition of an H-space as a product of two non-contractible spaces $X \simeq Y\times Z$ determines a non-trivial idempotent 
	$e = jp \colon X \to Y \hookrightarrow X$ 
	in $\End(X)$, and conversely, every non-trivial idempotent $f\in \End(X)$ gives rise to a non-trivial product decomposition 
	$X\simeq\Tel(f) \times \Tel(f \ldiv \id_X)$. 
	
	\begin{definition}
		An H-space $X$ is \emph{indecomposable} if $\0_X$ and $\id_X$ are the only idempotents in $\End(X)$. Moreover $X$ 
		is \emph{strongly indecomposable} if $\End(X)$ is a local loop near-ring. 
	\end{definition}	
		
	By theorem \ref{thm local} every strongly indecomposable H-space is indecomposable. The converse is not true: e.g. 
	$\End(S^1)=\End(S^3)=\End(S^7)\cong\ZZ$, so $S^1,S^3$ and $S^7$ are indecomposable H-spaces but they are not strongly indecomposable
	since the ring of integers is not local. The main result of this paper is that the distinction between indecomposable and 
	strongly indecomposable disappears when one considers finite $p$-local spaces.
	\begin{example}
%
		In the sense of Baker and May, see~\cite[definition 1.1]{BakerMay}, a $p$-local CW-complex or spectrum $X$ is called {\em atomic} 
		if its first nontrivial homotopy group, say $\pi_{k_0}(X)$, is a cyclic $\plocal$-module, and a self map $f \colon X \to X$ is 
		a homotopy equivalence whenever $f_\ihom \colon \pi_{k_0}(X) \to \pi_{k_0}(X)$ is an isomorphism. This notion of atomicity also appeared 
		earlier in~\cite[\S 4]{CMN}. Note that in this case 
		$\End(\pi_{k_0}(X))$ is a local ring, and the loop near-ring homomorphism $\pi_{k_0} \colon \End(X) \to \End(\pi_{k_0}(X))$ 
		is unit-reflecting. Hence, every atomic complex $X$ in this sense is also strongly indecomposable by theorem~\ref{thm:local_lnr}. 
		
		In particular, the spectra $BP$, $BP\langle n \rangle$ are atomic at all primes~\cite[examples 5.1, 5.4]{BakerMay}, suspensions 
		$\Sigma \mathbb{C}\mathrm{P}^\infty$, $\Sigma\mathbb{H}\mathrm{P}^\infty$ are atomic at the prime $2$ 
		by~\cite[propositions 4.5, 5.9]{BakerMay}. Moreover, at 
		the prime $p$, there is a decomposition~\cite[proposition 2.2]{McGibbon}
		\begin{equation}
			\Sigma \mathbb{C}\mathrm{P}^\infty_{(p)} \simeq W_1 \vee \cdots \vee W_{p-1} \textrm{, } \label{decSCP}
		\end{equation}
		where the nonzero integral homology groups of $W_r$ are $\widetilde{H}_{2k+1}(W_r) = \plocal$ for $k \equiv r \mod (p-1)$. 
		By~\cite[proposition 5.9]{BakerMay} the suspension spectra $\Sigma^\infty W_r$ are atomic, hence strongly indecomposable by dual 
		reasoning in view of remark~\ref{rem:dual1}.
		The loop near-ring homomorphism $\Sigma^\infty \colon \End(W_r) \to \End(\Sigma^\infty W_r)$ is unit-reflecting, so the coH-spaces 
		$W_r$ are also strongly indecomposable. Therefore, the $\vee$-decomposition~(\ref{decSCP}) is unique by the dual 
		of theorem~\ref{thm:KSAfin} below.
	\end{example}
	
	\begin{lemma} \label{non-trivial idempotent}
		Let $X$ be an H-space and let $f \in \End(X)$ be an idempotent. Then $f = \0_X$ if and only if $\beta_X(f) = 0$.
	\end{lemma}
	\begin{proof}
		It is the `if' part that requires a proof. Assume $\beta_X(f) = 0$ and 
		let $g$ solve the equation $g + f = \id_X$ in $\End(X)$. Then $\beta_X(g) = 1$, so $g$ is a homotopy equivalence by 
		proposition~\ref{prop beta}. Using right distributivity in $\End(X)$ we obtain $f = (g+f)f = gf + f$. Canceling $f$ we get $gf = \0_X$, 
		hence $f = \0_X$, since $g$ is a homotopy equivalence. 
	\end{proof}
	Lemma~\ref{non-trivial idempotent} combined with theorem \ref{thm local} yields the following detection principle.
%
	\begin{proposition} \label{prop detect}
	Let $X$ be an H-space.
		\begin{enumerate}[(i)]
			\item $X$ is indecomposable if and only if the ring $\im\beta_X$ contains no proper non-trivial idempotents.
			\item $X$ is strongly indecomposable if and only if the ring $\im\beta_X$ is local.
		\end{enumerate}
	\end{proposition}
	
	Let $X_i$ be H-spaces, set $X := \prod_{i\in I} X_i$, and equip $X$ with the H-space structure induced by the $X_i$. 
	Then $\End(X) = [X,X] = \prod_{i \in I} [X,X_i]$ as loops. Denote by $e_i \in \End(X)$ the idempotent 
	$j_i p_i \colon X \to X_i \hookrightarrow X$ corresponding to the factor $X_i$. As a loop, $[X,X_i]$ 
	is naturally isomorphic to $e_i\End(X)$, the isomorphism being given by $[X,X_i] \to e_i\End(X)$, $f \mapsto j_i f$. Therefore  
	$\End(X) \cong \prod_{i \in I} e_i\End(X)$. Setting $A := \im \beta_X$, it is easily seen 
	that $A = \prod_{i \in I} e_{i\ihom}A$, not only as abelian groups, but also as right $A$-modules. We shall exploit this fact on 
	multiple occasions, as it translates a decomposition problem of an H-space into a (seemingly) more manageable decomposition problem 
	of a module.
	\begin{remark}
		More can be said. The loop $[X,X_i]$ has a natural right action of the loop near-ring $\End(X)$ given by composition
		\[
			[X,X_i] \times \End(X) \to [X,X_i] \textrm{, } (f,h) \mapsto fh \textrm{. }
		\] 
		Naturality of the addition on $[X, X_i]$ implies that $(f+g)h = fh + gh$ holds, i.e. this action is right distributive over 
		$+$ and makes $[X, X_i]$ into an {\em $\End(X)$-comodule} (see~\cite[definition 13.2]{Clay}). The isomorphism $[X,X_i] \cong e_i \End(X)$ 
		is then an isomorphism of right $\End(X)$-comodules. Of course, once the functor $\pi_*$ is applied to 
		$\End(X) = \prod_{i \in I} [X,X_i]$, we obtain the aforementioned identification of right $A$-modules $A = \prod_{i \in I} e_{i\ihom}A$.
	\end{remark}
	The next technical lemma draws a tight relation between certain retracts of $X$ and corresponding summands of the right $A$-module $A$.
	\begin{lemma} \label{lemma:tech}
		Assume that $Z$ and $Z'$ are retracts of an H-space $X$, with $Z$ strongly indecomposable.
		Set $A := \im \beta_X$, and let $e_\ihom = (jp)_\ihom$ and $e'_\ihom = (j'p')_\ihom$ be the idempotents corresponding to 
		retracts $Z$ and $Z'$, respectively. Then $Z$ and $Z'$ are homotopy equivalent spaces if and only if $e_\ihom A$ and 
		$e'_\ihom A$ are isomorphic right $A$-modules.
	\end{lemma}
	\begin{proof}
		Suppose $Z \simeq Z'$. Pick a homotopy equivalence $v \colon Z \to Z'$ with homotopy inverse $v^{-1} \colon Z' \to Z$. 
		Consider the elements $(j'vp)_\ihom$ and $(jv^{-1}p')_\ihom$ in the ring $A$. 
		Note that $(jv^{-1}p')_\ihom(j'vp)_\ihom = e_\ihom$ and $(j'vp)_\ihom(jv^{-1}p')_\ihom=e'_\ihom$. 
		For any idempotent $f_\ihom \in A$ left multiplication by $f_\ihom$ is the identity of the right $A$-module $f_\ihom A$. 
		It follows that left multiplication by $(j'vp)_\ihom$ is an endomorphism of the right $A$-module $A$, which maps $e_\ihom A$ isomorphically 
		onto the submodule $e'_\ihom A$. Hence, $e_\ihom A \cong e'_\ihom A$.
		
		For the reverse implication, observe that $e_\ihom A e_\ihom$ and $\im \beta_{Z}$ are 
		isomorphic as rings, the latter ring being local by proposition~\ref{prop detect}. Since $e_\ihom A \cong e'_\ihom A$ 
		as right $A$-modules, the idempotents $e_{\ihom}$ and $e'_{\ihom}$ are conjugate in $A$, i.e. 
		$e'_{\ihom} = u^{-1}_{\ihom} e_{\ihom} u_{\ihom}$ 
		for some unit $u_{\ihom} \in U(A)$, see~\cite[exercise 21.16]{Lam}. Now form the composed maps
		\begin{align*}
			g = p u j' &\colon Z' \hookrightarrow X \to X \to Z \\ \textrm{and } \quad 
			h = p' u^{-1} j &\colon Z \hookrightarrow X \to X \to Z' \textrm{, } 
			\phantom{\quad \textrm{ and}}
		\end{align*}
		and verify that $g h$ and $h g$ induce the identity endomorphisms of the respective homotopy groups. 
		Therefore, $Z \simeq Z'$.
	\end{proof}
	Finite product decompositions of H-spaces behave nicely, as one is tempted to suspect from the module case.
	\begin{theorem} \label{thm:KSAfin}
		Assume that an H-space $X$ admits a (finite) product decomposition $X \simeq X_1 \times \cdots \times X_n$ into strongly indecomposable 
		factors $X_i$. Then:
		\begin{enumerate}[(i)]
			\item Any indecomposable retract $Z$ of $X$ is strongly indecomposable. Moreover, there is an index $i$ such that $Z \simeq X_i$. 
			\item If $X \simeq X'_1 \times \cdots \times X'_m$ is any other decomposition of $X$ into indecomposable factors $X'_k$, then 
				$m = n$, and there is a permutation $\varphi$ such that $X_i \simeq X'_{\varphi(i)}$ holds for all $i$.
		\end{enumerate}
	\end{theorem}
	\begin{proof}
		Set $A := \im \beta_X$. A retraction $p \colon X \to Z$ and its right inverse $j \colon Z \hookrightarrow X$ determine an 
		idempotent $f_\ihom = (jp)_\ihom$ in the ring $A$. Also, we have idempotents $e_{i \ihom} = (j_i p_i)_\ihom \in A$ and 
		$e'_{k \ihom} = (j'_k p'_k)_\ihom \in A$ corresponding to 
		the factors $X_i$ and $X'_k$ respectively. Viewing $A$ as a right $A$-module, we see that (i) $f_\ihom A$ is a direct 
		summand of $A$, and (ii) $A$ admits two direct-sum decompositions
		\[
			A = e_{1\ihom}A \oplus \cdots \oplus e_{n\ihom}A = e'_{1\ihom}A \oplus \cdots \oplus e'_{m\ihom}A \textrm{, }
		\]
		The theorem will now follow almost directly from its algebraic analogues:
		\begin{enumerate}[(i)]
			\item By~\cite[lemma 2.11]{Facchini} 
				$f_\ihom A$ has a local endomorphism ring. Moreover, $f_\ihom A$ is isomorphic to some $e_{i\ihom}A$. Since 
				$\End_A(f_\ihom A) \cong f_\ihom A f_\ihom \cong \im \beta_Z$ as rings, $Z$ is strongly 
				indecomposable by proposition~\ref{prop detect}. Hence, by lemma~\ref{lemma:tech}, $Z \simeq X_i$.
%
			\item By proposition~\ref{prop detect} the $A$-modules $e_{i\ihom}A$ are indecomposable with local endomorphism 
				rings, and, the $A$-modules $e'_{k \ihom}A$ are indecomposable. By the Krull--Schmidt--Remak--Azumaya 
				theorem~\cite[theorem 2.12]{Facchini} there is a bijection $\varphi \colon \{1, \ldots, n\} \to \{1, \ldots, m\}$ such 
				that $e_{i\ihom}A$ and $e'_{\varphi(i)\ihom}A$ are isomorphic right $A$-modules. Now use lemma~\ref{lemma:tech} to conclude 
				$X_i \simeq X'_{\varphi(i)}$ for all $i = 1,\ldots, n$. \qedhere
		\end{enumerate}
	\end{proof}
	We will use the proof above as a prototypical example of use of lemma~\ref{lemma:tech} to deduce uniqueness of H-space decompositions 
	from uniqueness of module decompositions. The Krull--Schmidt--Remak--Azumaya theorem for modules, however, is a statement about direct-sum 
	decompositions of modules, and is false for general, i.e. infinite, direct-product decompositions, see~\cite[example 2.1]{Franetic}. 
	The following proposition is a very special case of~\cite[theorem 2.4]{Franetic} that will be used later in this section.
	\begin{proposition} \label{prop:KSAprod}
		Let $R$ be a proper subring of the rationals, $A$ an $R$-algebra, and $\{M_i : i \in I\}$ and $\{N_k:k\in K\}$ two 
		countable families of indecomposable $A$-modules, which are finitely generated as $R$-modules. Assume that 
		$\End_A(M_i)$ are local rings. 
		If there is an isomorphism $\prod_{i \in I} M_i \cong \prod_{k \in K} N_k$, then there exists a bijection $\varphi \colon I \to K$ 
		such that $M_i \cong N_{\varphi(i)}$ as $A$-modules.
	\end{proposition}
%

	Fix a {\em proper} subring $R \lneq \mathbb{Q}$. We will call a connected H-space $X$ {\em $R$-local} if $\pi_*(X)$ is an $R$-module. 
	A connected $R$-local H-space $X$ is called {\em homotopy-finite} if $\pi_*(X)$ is finitely generated over $R$, it is called 
	{\em of finite type} if $\pi_k(X)$ is finitely generated over $R$ for each $k$. 
		
	In~\cite{Gray} B. Gray proves a unique decomposition theorem for finite type H-spaces in the $p$-complete setting, 
	see~\cite[corollary 1.4]{Gray}. For $R$-local finite type H-spaces we have the following results 
	(theorems~\ref{thm:KSAprod} and~\ref{thm:KSAwprod}). 
	\begin{theorem}
		\label{thm:KSAprod}
		Let $\{X_i : i \in I\}$ and $\{X'_k : k \in K\}$ be two families of $R$-local, homotopy-finite H-spaces, with all of the $X_i$ 
		strongly indecomposable, and all of the $X'_k$ indecomposable. Assume that the product $\prod_{i\in I} X_i$ is of finite type. If 
		the products $\prod_{i\in I} X_i$ and $\prod_{k \in K} X'_k$ are homotopy equivalent, then there exists a bijection 
		$\varphi \colon I \to K$ such that $X_i \simeq X'_{\varphi(i)}$ for all $i$.
	\end{theorem}
	\begin{remark}
		More often than not, the products in the above statement will not have the homotopy type of a CW-complex, even though we are 
		assuming that the spaces $X_i$ and $X'_k$ are CW-complexes (or have the homotopy type of a CW-complex). 
	\end{remark}
	\begin{proof}
		We set $X := \prod_{i \in I} X_i$, $A := \im \beta_X$, and use $e_i = j_i p_i$ and $e'_k = j'_k p'_k$ to denote the idempotents in
		$\End(X)$ corresponding to factors of each decomposition. Then $A$ is an $R$-algebra and the right $A$-module $A$ admits two direct 
		product decompositions
		\[
			A = \prod_{i \in I} e_{i\ihom} A = \prod_{k \in K} e'_{k\ihom} A \textrm{. }
		\] 
		By proposition~\ref{prop detect} the $A$-modules $e_{i\ihom} A$ are strongly indecomposable, while the $A$-modules 
		$e'_{k\ihom}A$ are indecomposable.
		
		View $e_{i\ihom}A$ as an $R$-submodule of $\Hom_R(\pi_*(X), \pi_*(X_i))$ via the monomorphism 
		$e_{i\ihom}A \to \Hom_R(\pi_*(X), \pi_*(X_i))$, $e_{i\ihom} f_\ihom \mapsto p_{i\ihom} f_\ihom$. 
		Since $\pi_*(X_i)$ is finitely generated over $R$ and $X$ is of finite type, $\Hom_R(\pi_*(X), \pi_*(X_i))$---the $R$-module of 
		graded homomorphisms $\pi_*(X) \to \pi_*(X_i)$---is finitely generated. As $R$ is noetherian, each $e_{i\ihom} A$ must 
		also be finitely generated as an $R$-module. Similarly, each $e'_{k\ihom}A$ is also finitely generated as an $R$-module. 
		Now, $X$ of finite type forces both index sets, $I$ and $K$, to be at most countable. 
		Hence, all of the assumptions of propositon~\ref{prop:KSAprod} are satisfied, so there is a bijection $\varphi \colon I \to K$, 
		such that $e_{i\ihom} A$ and $e'_{\varphi(i)\ihom} A$ are isomorphic right $A$-modules. By lemma~\ref{lemma:tech} we must have 
		$X_i \simeq X'_{\varphi(i)}$ for all $i \in I$. 
	\end{proof}
	
	There is another decomposition of spaces often studied in homotopy category, the weak product. Let 
	$X'$ be the set of all points $x = (x_i)_{i \in I} \in \prod_{i \in I} X_i$ with all but finitely many of the $x_i$ equal to the 
	base point $*_i \in X_i$. Equip the product $\prod_{i \in I} X_i$ with the compactly generated topology, and let $X'$ inherit the 
	relative topology. We will (deliberately) use the notation $\wprod_{i \in I} X_i$ for the space $X'$ and call it the {\em weak 
	product of the $X_i$}. Of course, $X'$ can also be viewed as a union (direct limit) of all compactly generated {\em finite} products 
	of the $X_i$. Hence, if all of the $X_i$ are $T_1$-spaces, there is a natural isomorphism 
	$\pi_*(\wprod_{i \in I} X_i) \cong \bigoplus_{i \in I} \pi_*(X_i)$. Also, if all of the $X_i$ are CW-complexes, 
	then the topology on $\wprod_{i \in I} X_i$ is precisely the CW-topology. 
	
	Let $\{X_i : i \in I\}$ be a family of H-spaces with additions $\mu_i \colon X_i \times X_i \to X_i$. 
	Define a map $\mu' \colon X' \times X' \to X'$ to be the composite 
	\[ \textstyle
		X' \times X' = \left(\wprod_{i \in I} X_i\right) \times \left(\wprod_{i \in I} X_i\right)  \xrightarrow{\tau} 
		\wprod_{i \in I} (X_i \times X_i) \xrightarrow{\oplus_{i \in I} \mu_i} \wprod_{i \in I} X_i = X'\textrm{, }
	\]
	where $\tau$ is the coordinate shuffle map, i.e. $\tau((x_i)_{i \in I}, (y_i)_{i \in I}) = (x_i, y_i)_{i \in I}$. Clearly, 
	$\tau$ is well-defined. Continuity of $\tau$ is assured by equipping all the products above with the compactly generated topology. 
	A routine exercise shows that $\mu' j'_1 \simeq \id_{X'} \simeq \mu' j'_2$ holds for the two inclusions 
	$j'_1, j'_2 \colon X' \hookrightarrow X' \times X'$. (Let $j_{i1} \colon X_i \hookrightarrow X_i \times X_i$ be the inclusions 
	of the first factor, and suppose $H_i \colon X_i \times \II \to X_i$ are homotopies $\mathrm{rel}\ *_i$ from 
	$\mu_i j_{i1}$ to $\id_{X_i}$. Consider the composite 
	\[ \textstyle
		H' \colon X' \times \II \hookrightarrow \left(\prod_{i\in I} X_i\right) \times \II^I \xrightarrow{\tau} 
		\prod_{i \in I} (X_i \times \II) \xrightarrow{\Pi_{i \in I} H_i} \prod_{i\in I} X_i \textrm{, }
	\] 
	where $\II \hookrightarrow \II^I$ is the diagonal inclusion. Note that in fact $H'(X' \times \II) \subseteq X'$, 
	since the homotopies $H_i$ are $\mathrm{rel}\ *_i$. Therefore $H'$ is a homotopy $\mathrm{rel}\ (*_i)_{i\in I}$ 
	from $\mu'j'_1$ to $\id_{X'}$. Repeat for $j'_2$.) Hence, $\wprod_{i \in I} X_i$ is also an H-space.
	
	Let $X_i$ be H-spaces, and let $X' = \wprod_{i \in I} X_i$ be the weak product of the $X_i$. Again,  
	we denote by $e_i = j_i p_i$ the idempotent in $\End(X')$ corresponding to the factor $X_i$. The functor $\pi_*$ maps 
	$\End(X')=[X',X']$ into 
	\[
		\textstyle
		\End_R(\pi_*(X')) = \Hom_R(\bigoplus_{i \in I} \pi_*(X_i), \pi_*(X')) = \prod_{i \in I} \Hom_R(\pi_*(X_i), \pi_*(X')) \textrm{, }
	\]
	and there is a natural identification $\Hom_R(\pi_*(X_i),\pi_*(X')) \cong \End_R(\pi_*(X'))e_{i\ihom}$. Set $A' := \im \beta_{X'}$. 
	Restricting the decomposition above to the subring $A'$ of $\End_R(\pi_*(X'))$ we get $A' = \prod_{i \in I} A'e_{i\ihom}$
	as a {\em left} $A'$-module.
	
	We can now state the weak product version of theorem~\ref{thm:KSAprod}. The proof is deliberately omitted, as it uses the same  
	argument as the proof of theorem~\ref{thm:KSAprod} with the left $A'$-module $A'$ in place of the right $A$-module $A$.
	\begin{theorem}
		\label{thm:KSAwprod}
		Let $\{X_i : i \in I\}$ and $\{X'_k : k \in K\}$ be two families of $R$-local, homotopy-finite H-spaces, with all of the $X_i$ 
		strongly indecomposable, and all of the $X'_k$ indecomposable. Assume that the weak product $\wprod_{i\in I} X_i$ is of finite type. 
		If the weak products $\wprod_{i\in I} X_i$ and $\wprod_{k \in K} X'_k$ are homotopy equivalent, then there exists a bijection 
		$\varphi \colon I \to K$ such that $X_i \simeq X'_{\varphi(i)}$ for all $i$.
	\end{theorem}
	
	Of course, the above uniqueness theorems say nothing about the existence of factorizations of H-spaces as products of strongly 
	indecomposable spaces. 	
	For example if $X$ is an H-space having the homotopy type of a finite CW-complex, then one can decompose $X$ as 
	a product of indecomposable factors
	but these factors will rarely be strongly indecomposable (unless $\End(X)$ is finite). This is reflected in the well-known phenomenon 
	that finite H-spaces 
	often admit non-equivalent product decompositions. The situation becomes more favorable if we consider $p$-localizations of  
	H-spaces. In the following section we are going to 
	show that a $p$-local finite H-space is indecomposable if and only if it is strongly indecomposable. 
	A factorization of an H-space as a product of such spaces
	is therefore unique. Finally, if we consider $p$-complete H-spaces then even the finite-dimensionality  assumption may be dropped. 
	In fact, Adams and Kuhn~\cite{Adams-Kuhn} 
	have proved that every indecomposable $p$-complete H-space of finite type is \emph{atomic}, which in particular implies, that it is 
	strongly indecomposable. 
	Theorem~\ref{thm:KSAfin} implies that decompositions into finite products of $p$-complete atomic 
	spaces are unique. For an alternative approach that works for spaces of finite type see \cite[corollaries 1.4 and 1.5]{Gray} or 
	\cite[theorem 4.2.14]{Xu}.

\section{Homotopy endomorphisms of $p$-local spaces}
\label{sect:p-local}
	In this section we consider $p$-local H-spaces and show that under suitable finiteness assumptions the indecomposability of 
	a space implies strong indecomposability. 	
	The proof is an interesting blend of topology and algebra, since it uses non-trivial results from homotopy theory, 
	the theory of local rings and the theory of loop near-rings. 
	Let us say that an $R$-local H-space $X$ is \emph{finite} if it is finite-dimensional, and if its homotopy groups are 
	finitely generated $R$-modules for some subring $R \le \QQ$. For every finite  H-space $X$ we can define the homomorphism 	
	$$	\bar\beta_X\colon \End(X)\to  \prod_{k=1}^{\dim X} \End(\pi_k(X)), \quad f\mapsto (\pi_1(f), \pi_2(f), 
	\ldots, \pi_{\dim X}(f)) \textrm{. }$$
	We will show that---when applied to a finite H-space $X$---the homomorphism $\bar\beta_X$ retains the same main features of the 
	homomorphism $\beta_X$ as described in proposition \ref{prop beta}, 
	while it has a great advantage over the latter because it maps into the ring of endomorphisms of a finitely generated module.
	 
	\begin{proposition}
		\label{prop:findim}
		If $X$ is a finite H-space then the homomorphism $\bar\beta_X$ is unit-reflecting and idempotent-lifting.
	\end{proposition}	
	\begin{proof}
		Reflection of units follows from the Whitehead theorem, so it only remains to prove that $\bar\beta_X$ is idempotent lifting.

		First observe that finite H-spaces are rationally elliptic, i.e. $X$ is rationally equivalent to a finite product of 
		Eilenberg--MacLane spaces; 
		$X_\QQ \simeq K(\QQ, n_1) \times \cdots \times K(\QQ, n_t)$, see~\cite[section 4.4]{Zabrodsky}.  It follows that for all 
		$k > \dim X$ the groups $\pi_k(X)$ are torsion and hence finite. 
		
		Let a map 
		$f \colon X \to X$ be such that $\pi_k(f) = \pi_k(f)^2 \colon \pi_k(X) \to \pi_k(X)$ for all $k \le \dim X$, i.e. 
		$\bar\beta_X(f)$ is an idempotent in $\im \bar\beta_X$. As the groups $\pi_k(X)$ are finite for $k > \dim X$, there is an 
		integer $n$  such that the $n$-fold composite $f^n \colon X \to X$ induces an idempotent endomorphism 
		$\pi_k(f)^n \colon \pi_k(X) \to \pi_k(X)$ for all $k \le 2(\dim X+1)$. 
		If we set $\bar{f}:= f^n \ldiv \id_X$ in the loop near-ring $\End(X)$, then $\pi_k(\bar{f}) = \id_{\pi_k(X)} - \pi_k(f)^n$ is 
		an idempotent endomorphism of $\pi_k(X)$ for all $k \le 2(\dim X+1)$. It follows that the map
		\[
			X \xrightarrow{\Delta} X \times X \hookrightarrow \Tel(f^n) \times \Tel(\bar{f})
		\]
		induces an isomorphism
		\[
			\pi_k(X) \to \im \pi_k(f)^n\oplus \im \pi_k(\bar{f})
		\]
		for all $k \le 2(\dim X+1)=\dim(\Tel(f^n) \times \Tel(\bar{f}))$. Hence, $X \simeq \Tel(f^n)\times \Tel(\bar{f})$ by the 
		Whitehead theorem. 
		This product decomposition determines the idempotent $e\colon X\to \Tel(f^n)\to X$ that satisfies 
		$\pi_k(e) = \pi_k(f)^n = \pi_k(f)$ for all $k \le \dim X$. 
		In other words $\bar\beta_X(e)=\bar\beta_X(f)$, therefore $e$ is an idempotent in $\End(X)$ that lifts $\bar\beta_X(f)$. 
	\end{proof}

	We have now prepared all the ingredients needed for the proof of the main result of this section.
 
	\begin{theorem}
		\label{thm fin}		
		Every indecomposable finite $p$-local H-space is strongly indecomposable.
	\end{theorem}
	\begin{proof}
		To simplify the notation, let us denote by $E$ the ring $\prod_{k=1}^{\dim X} \End(\pi_k(X))$, by $A$ its subring $\im\bar{\beta}_X$, 
		and by $J = J(A)$ the Jacobson radical of $A$. 
		
		Theorem~\ref{thm local} says that in order to prove that $\End(X)$ is a local loop near-ring 
		it is sufficient to show that $A$ is a local ring. The ring $A$ is finitely generated as a $\plocal$-module, so 
		by~\cite[proposition 20.6]{Lam} the quotient $A/J$ is semisimple (i.e. a product of full-matrix rings over division
		rings). Therefore, we must prove that $A/J$ has only trivial idempotents, as this would imply that $A/J$ is a division ring, and 
		hence that $A$ is local. In fact, it is sufficient to prove that $J$ is an idempotent-lifting ideal because
		then every non-trivial idempotent in $A/J$ would lift to a non-trivial idempotent in $A$, and then along $\bar\beta_X$ 
		to a non-trivial idempotent in $\End(X)$, contradicting the indecomposability of $X$.
		
		That $J$ is idempotent-lifting is proved by the following argument. The ring $E$ is semiperfect 
		by~\cite[examples 23.2 and 23.4]{Lam} because it is a product of endomorphism rings of finitely generated $\ZZ_{(p)}$-modules. 
		By \cite[lemma 3.2]{FraPav} $A$ is a subring of finite additive index in $E$, and so by \cite[example 3.3]{Pav} the radical
		$J$ is idempotent-lifting, which concludes the proof. 
	\end{proof}
	
	Let us remark that if $X$ is a $p$-local H-space whose graded homotopy group is a finitely generated $\ZZ_{(p)}$-module
	(i.e. $X$ is a homotopy-finite $p$-local H-space) then the above proof works with $\beta_X$ in place 
	of $\bar\beta_X$, and we obtain the following result as well.
	\begin{theorem}
		\label{thm finhmtp}
		Let $X$ be a $p$-local H-space such that its graded homotopy group is a finitely generated $\ZZ_{(p)}$-module. Then $X$
		is indecomposable if and only if it is strongly indecomposable. 
	\end{theorem}
	\begin{remark}\label{rem:dual2}
		In case of simply-connected $p$-local coH-spaces (or $p$-local connective CW-spectra) $X$ there is no distinction between {\em finite} and 
		{\em homology finite} (at least up to homotopy equivalence). Theorems \ref{thm fin} and \ref{thm finhmtp} are therefore replaced by one 
		dual theorem. In the proof of theorem~\ref{thm fin} we simply replace the homomorphism $\bar\beta_X$ with $\alpha_X$ without any 
		additional complications. No dual of proposition~\ref{prop:findim} is needed.
	\end{remark}
	
	Observe that the two versions of the  theorem of Wilkerson~\cite{Wilkerson} on the unique factorization of $p$-local H-spaces now follow as
	easy corollaries. In fact, every $p$-local H-space of finite type that is either finite-dimensional or homotopy finite-dimensional
	admits a decomposition as a product of indecomposable factors. By theorems \ref{thm fin} and \ref{thm finhmtp} the factors are indeed strongly 
	indecomposable, so by theorem \ref{thm:KSAfin} the decomposition is unique.	
	\begin{example}
		One might wonder whether the theorems~\ref{thm fin} and~\ref{thm finhmtp} remain true if we replace {\em finite} by {\em finite type} 
		(i.e. $\pi_k(X)$ are finitely generated for all $k$). We know that at least in the case of CW-spectra they are false. Consider the example 
		given by Adams and Kuhn in~\cite[\S 4]{Adams-Kuhn}. They construct an indecomposable $p$-local spectrum $X$, such that the ring homomorphism 
		$H_0 \colon \End(X) \to \End(H_0(X))$ is unit-reflecting and its image is a ring isomorphic to 
		\[
			\frac{\plocal[\lambda]}{(\lambda^2-\lambda+p)} \textrm{. }
		\]
		The spectrum $X$ has $H_0(X) = \plocal \oplus \plocal$, so we can identify
		$\End(H_0(X))$ with $M_2(\plocal)$, the ring of $2 \times 2$ matrices with entries in 
		$\plocal$. The image of $\End(X)$ in this matrix ring is precisely the one-to-one image of 
		$\plocal[\lambda]/(\lambda^2-\lambda+p)$ under the ring homomorphism which maps a polynomial $q$ to the matrix $q(A)$, where 
		\[
			A = \begin{pmatrix} 0 & 1 \\ -p & 1 \end{pmatrix} \textrm{. }
		\] 
		(Note that $\lambda^2 - \lambda +p$ is the minimal polynomial of $A$.) Now, $A$ is not invertible, and neither is $I - A$, 
		so the ring $\plocal[\lambda]/(\lambda^2-\lambda+p)$ cannot be local. Hence, $X$ is an indecomposable $p$-local spectrum 
		of finite type, which is not strongly indecomposable. 
		
		Adams' and Kuhn's construction of the spectrum $X$ relies on the existence of certain elements in the stable homotopy groups of spheres 
		(in the image of $J$-homo\-mor\-phism) and cannot be directly applied to spaces. It remains an open question whether a similar example 
		exists in the realm of finite type H- or coH-spaces.
	\end{example}
\bibliographystyle{plain}
\bibliography{ref}
\end{document}